\documentclass[a4paper,reqno]{amsart}
\pdfoutput=1

\usepackage{amsmath,amsfonts,amssymb,amsthm,stmaryrd,bbm,graphicx,mathtools,enumerate}
\usepackage{mathrsfs}
\usepackage{cancel}
\usepackage[T1]{fontenc} 
\usepackage[utf8]{inputenc}
\usepackage[english]{babel}
\usepackage[tracking,spacing,kerning,babel]{microtype}
\usepackage{wasysym} 
\usepackage{color}
\usepackage{xcolor}
    	\usepackage{framed} 
\usepackage{wrapfig} 
\usepackage{ulem}
\mathtoolsset{showonlyrefs}
\usepackage[final]{hyperref}   
\hypersetup{
    linktoc=page,
    linkcolor=red,          
    citecolor=blue,        
    filecolor=blue,      
    urlcolor=cyan,
   colorlinks=true           
}

\usepackage{lipsum} 

\usepackage{minitoc}
\usepackage{multicol}

\newtheoremstyle{mine}
{\baselineskip}
{\baselineskip}
{\itshape}
{
}
{\bfseries}
{.}
{.5em}
{#1 #2\ifx#3\relax\else~(#3)\fi}

\theoremstyle{mine}

\newtheorem{theorem}{Theorem}
\newtheorem{corollary}[theorem]{Corollary}
\newtheorem{proposition}[theorem]{Proposition}
\newtheorem{lemma}[theorem]{Lemma}
\newtheorem{definition}[theorem]{Definition}
 
\theoremstyle{remark}
\newtheorem{remark}{Remark}

\colorlet{shadecolor}{blue!10}

\renewcommand{\epsilon}{\varepsilon}

\newcommand{\Z}{\mathbb{Z}}

\def\<#1{\langle #1\rangle}

\def\bi{\begin{itemize}}  
\def\ei{\end{itemize}}
\def\bnum{\begin{enumerate}} 
\def\enum{\end{enumerate}}

\title[Long-range models in 1D  revisited]
{
Long-range models in 1D revisited
}

\author{Hugo Duminil-Copin, Christophe Garban, Vincent Tassion}

\address
{Universit\'e de Gen\`eve, 2-4 rue du Li\`evre, 1204 Gen\`eve, Switzerland, Institut des Hautes \'Etudes Scientifiques, 35 route de Chartres, 91440 Bures-sur-Yvette, France}
\email{hugo.duminil@unige.ch,duminil@ihes.fr}

\address
{Université Claude Bernard Lyon 1, CNRS UMR 5208, Institut Camille Jordan, 69622 Villeurbanne, France \, and Institut Universitaire de France (IUF)}
\email{garban@math.univ-lyon1.fr}

\address
{ETH Zurich, Department of Mathematics, Group 3
HG G 66.5
Rämistrasse 101,
8092 Zurich,
Switzerland}
\email{Vincent.Tassion@math.ethz.ch}

\begin{document}

\maketitle

\begin{abstract}
In this short note, we revisit a number of classical result{s} on long-range 1D percolation, Ising model and Potts models \cite{FS82,NewSch,aizenman1988,ImbrieNewman}. 
More precisely, we show that for Bernoulli percolation, FK percolation and Potts models, there is symmetry breaking for the $1/r^2$-interaction at large $\beta$, and that the phase transition is necessarily discontinuous. {We also show, following the notation of \cite{aizenman1988} that $\beta^*(q)=1$ for all $q\geq 1$. }
\end{abstract}

\section{Setting}

Long-range models on the 1D line have a rich history in physics and mathematical physics.  Historically, Dyson, motivated by predictions of Anderson \cite{AYH70} and Thouless \cite{Tho69} as well as connections with the {\em Kondo problem}, initiated in \cite{Dys69} the rigorous analysis of long-range Ising models on $\Z$ with coupling constants given by $J_{i-j}\sim |i-j|^{-s}$ with $s\in(1,2)$. 
 Fröhlich and Spencer analysed in \cite{FS82} the ``scale-invariant'' case where the coupling constants are given by $\frac 1 {|i-j|^2}$. 
Later,  Aizenman, Chayes, Chayes and Newman proved in \cite{aizenman1988} the discontinuity of the phase transitions which had been anticipated by Thouless for all $q\geq 1$. In this paper, we revisit these classical results.

We begin by treating the case of Bernoulli percolation and then consider the random-cluster model and its applications to the Ising and Potts models. Note that historically, the Ising model was studied before Bernoulli percolation, but our renormalization technique is simpler to present in the Bernoulli percolation setting.

Our notation/setup will follow in part \cite{NewSch,aizenman1988}.  In the whole paper, we consider $J_{i,j}=J(i-j)=1/|i-j|^2$ for every $i\ne j$. For $\beta,\lambda>0$, we define 
\begin{equation}
p_{i,j}(\beta,\lambda):=\begin{cases}1-\exp[-\beta J_{i,j}]&\text{ if }|i-j|\ge 2,\\ 1-\exp[-\lambda]&\text{ if }|i-j|=1.\end{cases}
\end{equation}
For simplicity, we will only consider the most interesting case where the point-to-point interaction will decay as $1/|i-j|^2$ but our methods   allows one to recover  known results for more general weights, e.g. $J(x)=1/x^s$, $s\in (1,2)$, and  other percolation models, such as   directed percolation, see Remark~\ref{rmk:1}.  

Let us point out that the renormalization argument in this paper follows the same setup as in our work \cite{DGT20} except that the $1D$ setting here makes things simpler.

\section{Long-range Bernoulli percolation}

\subsection{Statement of the result.}

Consider the long-range Bernoulli percolation measure $\mathbb P_{\beta,\lambda}$ on $\mathbb Z$ defined by the property that each unordered pair $\{i,j\}$ (also called edge) is {\em open} with probability $p_{i,j}(\beta,\lambda)$, and {\em closed} with probability $1-p_{i,j}(\beta,\lambda)$, independently for every edge $\{i,j\}\subset \mathbb Z$. Let $\theta(\beta,\lambda)$ be the probability that $0$ is connected to infinity by a path of open edges.

\begin{theorem}\label{thm:main perco}
We have the following two properties:
\begin{itemize}
\item[(i)] For $\beta>1$, $\theta(\beta,\lambda)>0$ for $\lambda<\infty$ large enough.
\item[(ii)] For every $\beta,\lambda>0$, $\theta(\beta,\lambda)>0$ implies $\beta\theta(\beta,\lambda)^2\ge 1$.
 \end{itemize}
\end{theorem}

The fact that for large $\beta$, $(i)$ is true was obtained by Newman and Schulman in \cite{NewSch}. The extension to every $\beta>1$ was proved in  \cite{ImbrieNewman}. On the other hand, (ii) is the object of  \cite{AN86}. Note that (ii) implies that $\theta(\beta,\lambda)=0$ if $\beta\le 1$, and that for each fixed $\lambda>0$, $\beta\mapsto\theta(\beta,\lambda)$ is continuous only if it is identically equal to 0 (hence the conclusion that the phase transition is necessarily discontinuous  when it occurs). {As such, using the notations in \cite{aizenman1988}, by combining $(i)$ and $(ii)$, this proves $\beta^*(q=1)=1$.}


\subsection{Notation.}
\label{sec:notation}

In the proofs below, we will use the notion of $K$-{\em block} \[B^i_K=[K(i-1),K(i+1)),\quad i\in\mathbb Z,\,K\in \mathbb Z_+.\] For simplicity we write $B_K$ instead of $B^0_K$. {Let us point out that in this framework, consecutive blocks $B_K^i$ and $B_K^{i+1}$ are overlapping on half of their length. This will be a key property in the proof below.}
Let $S\subset \mathbb Z$. We call a {\em cluster in $S$} a connected component $C\subset S$ of the graph with vertex set $S$ and open edges with {\em both} endpoints in $S$. 

A $K$-block $B_K^i$  is is said to be $\theta$-{\em good} if there exists a cluster in it of cardinality at least $2\theta K$.   When a block is not $\theta$-good, we call it $\theta$-{\em bad} and we define
  \begin{equation*}
    p_{\beta,\lambda}(K,\theta):=\mathbb P_{\beta,\lambda}[\text{$B_K$ is $\theta$-bad}].
  \end{equation*}

\subsection{Proof of Theorem~\ref{thm:main perco}(i).}


The proof of Theorem~\ref{thm:main perco}(i) relies on the idea that clusters at scale $K$ and local density $\theta$ will merge and with high probability create  new clusters at scale $CK$ of local density $\theta'=\theta-O(1/C)$ slightly smaller than $\theta$ (this slight loss of density allows us to lose a few clusters at scale $K$ in the process). More precisely, we prove the following renormalization inequality.

\begin{lemma}\label{lem:1} Let $\beta>1$ and  $\theta_\infty\in(\tfrac34,1)$ satisfying $\theta_\infty^2\beta>1$. There exist $C_0\ge1$  large enough (depending on $\theta_\infty,\beta$) such that the following holds. 
For every $\lambda>0$, $\theta\ge \theta_\infty$, and for every integers $C\ge C_0$ and $K\ge 2$,
\begin{equation}\label{eq:1}
p_{\beta,\lambda}(CK,\theta - C_0/C)\le \tfrac1{100}\,p_{\beta,\lambda} (K,\theta)+ 2C^2\,p_{\beta,\lambda}(K,\theta)^2.
\end{equation}
\end{lemma}

\begin{proof} In the proof, we focus on the $K$-blocks included in $B_{CK}$. Given such a block $B_K^i$, we write $\mathbf C(B^i_K)$ for the largest cluster in $B^i_K$. Notice that $\mathbf C(B^i_K)$ has size at least $2\theta K$  when  the block is $\theta$-good,  and it is the unique cluster in $B^i_K$ with this property when  $\theta>3/4$.

Let $C_0\ge6$  be a large constant to be chosen later (more precisely, the constant  $C_0$ will be chosen in such a way that  the second inequality of \eqref{eq:7} holds, and this choice depends on $\beta$ and  $\theta_\infty$ only) and set $\theta':=\theta-C_0/C$.  For $|i|\le C$, let $E_i$ be the event that $B_K^i$ is $\theta$-bad and all the blocks $B_K^j$ are $\theta$-good for $j\in [-{C+1},{C-1}]\setminus\{i-1,i,i+1\}$, and set
  \begin{equation}
    \label{eq:9}
    F_i=E_i\cap \{B_{CK}\text{ is $\theta'$-bad}\}.
  \end{equation}
   Observe that if all $K$-blocks $B^j_K$,  {$-C+1\le j \le C-1$},  are $\theta$-good, then the assumption that $\theta>3/4$ and the overlapping property between subsequent $K$-blocks guarantees that all the clusters $\mathbf C(B^j_K)$ are connected together in $B_{CK}$, which implies the existence of a cluster in $B_{CK}$ with cardinality larger than $2 \theta CK$. In particular, if $B_{CK}$ is $\theta'$-bad, then either there exist (at least)
    two disjoint $\theta$-bad $K$-blocks, or there exists $i$ such that {$F_i$} occurs.  The union bound implies
  \begin{equation}
    \label{eq:6}
    p_{\beta,\lambda}(CK,\theta')\le  \sum_{i=-{C+1}}^{{C-1}}\mathbb P_{\beta,\lambda}[F_i]+ \mathbb P_{\beta,\lambda}[\text{there are two {\em disjoint} $\theta$-bad $K$-blocks}].
  \end{equation}
By independence and the union bound, we have 
\begin{equation}\label{eq:1a}
\mathbb P_{\beta,\lambda}[\text{there are two {\em disjoint} $\theta$-bad $K$-blocks}]\le  \binom{2C-1}{2}p_{\beta,\lambda}(K,\theta)^2.
\end{equation}

  It remains to bound the first term in \eqref{eq:6}, which is the object of the end of the proof. If all $K$-blocks $B^j_K$ with $|j|\le C-C_0$  are $\theta$-good, the same argument as above implies that $B_{CK}$ is $\theta'$-good, therefore $F_i=\emptyset$ whenever $|i|> C-C_0$. 
   
 Now, let $|i|\le C-C_0$. Since $\mathbb P_{\beta,\lambda}[E_i]\le  p_{\beta,\lambda}(K,\theta)$, we deduce that
 \begin{align}
   \label{eq:10}
   \mathbb P_{\beta,\lambda}[F_i]\le p_{\beta,\lambda}(K,\theta)  \cdot \mathbb P_{\beta,\lambda}[B_{CK}\text{ is $\theta'$-bad}\:| \: E_i].
 \end{align}
 In order to bound the conditional probability above,  define $\mathbf C^-$ (resp.  $\mathbf C^+$) as the union of all the clusters $\mathbf C(B^j)$, $j\le i-2$ (resp.~$j\ge i+2$). Let us examine the properties of these two sets when the event $E_i$ occurs.

 First, the goodness property of the $K$-blocks at the left of $B_K^{i-1}$ and the right of $B_K^{i+1}$ imply that $\mathbf C^-$ and $\mathbf C^+$ are two connected sets and their sizes satisfy
   \begin{equation}
     \label{eq:19}
     |\mathbf C^-|,|\mathbf C^+|\ge \theta K(C-|i|-2) \ge \frac{K}2(C-|i|), 
   \end{equation}
   where we use  $C-|i|\ge C_0\ge 6$ and $\theta\ge \tfrac34$ to get  the second inequality.
 
 Second,  the $\theta$-density in each $K$-block implies that
   \begin{equation}
     |\mathbf C^+\cap [Ki+K, Ki+K+ 2K\ell) |\ge 2K\theta\ell \label{eq:22}
   \end{equation}
   for every integer  $\ell\ge1$. Write  $Ki<y_1<y_2<\cdots$ for the ordered elements of $\mathbf C^+$. Let $b\ge 1$. By applying the equation above to $\ell=\lceil
   \tfrac b{2K\theta}\rceil$, we obtain that
   \begin{equation}
     \label{eq:23}
     y_b\le Ki+K+ 2K \lceil \tfrac b{2K\theta}\rceil \le Ki+ \frac{3K+b}{\theta}.
   \end{equation}
   Equivalently, writing $Ki>x_1>x_2>\cdots$   for the ordered elements of $\mathbf C^-$, we have  
    $x_a\ge Ki-\frac {3K+a}{\theta}$   for every $a\ge 1$.  Therefore, for every $a,b\ge 1$, we have 
     \begin{equation}
       \label{eq:21}
       \theta( y_b-x_a)\le{6K +a+b}.
     \end{equation}
 Furthermore, if $B_{CK}$ is $\theta'$-bad, then $\mathbf C^-$ and $\mathbf C^+$ cannot be connected together. Conditioning on $\mathbf C^-$ and $\mathbf C^+$ provides no information on edges $\{x,y\}$ with $x\in\mathbf C^-$ and $y\in\mathbf C^+$ since the definition of $\mathbf C^-$ and $\mathbf C^+$ involves only edges with endpoints within a distance $2K$ of each other. Therefore, the conditional probability that $\mathbf C^-$ and $\mathbf C^+$ are {\em not} connected by an edge is equal to
 \begin{align}
  P(\mathbf C^-,\mathbf C^+):= \prod_{x\in \mathbf C^-}\prod_{y\in\mathbf C^+}e^{-\beta J_{x,y}}=\exp\Big[-\beta \theta^2  \sum_{a=1}^{|\mathbf C^-|}\sum_{b=1}^{|\mathbf C^+|}\frac1{\theta^2(y_b-x_a)^2}\Big].\label{eq:24}
 \end{align}
Set $A:=\tfrac{K}2(C-|i|)$. By using the bounds  \eqref{eq:19} and \eqref{eq:21}, and then a comparison sum-integral we find that 
    \begin{align}
      \label{eq:20}
      \sum_{a=1}^{|\mathbf C^-|}\sum_{b=1}^{|\mathbf C^+|}\frac1{\theta^2(y_b-x_a)^2}&\ge  
 \sum_{1\le a,b\le A}\ \frac1{(6K+a+b)^2} \nonumber\\& \ge  \int_{0\le x,y\le A}  \frac{dx\: dy}{(6K+x+ y  + 2)^2}\nonumber\\
                                                                                          &\ge \int_{0\le x,y\le A}  \frac{dx\: dy}{(8 K+x+ y)^2}\nonumber\\
  &=  \log\left(1+\frac{2A}{8K}\right)\ge \log \left(\frac{C -|i|}{8}\right).
    \end{align}
    Plugging this estimate in~\eqref{eq:24}, we obtain
    \begin{equation}
      \label{eq:25}
       P(\mathbf C^-,\mathbf C^+)\le \left(\frac 8{C-|i|}\right)^{\beta\theta^2},
    \end{equation}
    and then
integrating over all possible choices of $\mathbf C^-$ and $\mathbf C^+$, we get
\begin{equation}
  \label{eq:11}
  \mathbb P_{\beta,\lambda}[B_{CK}\text{ is $\theta'$-bad}\:| \: E_i]\le  \Big(\frac{8}{C-|i|}\Big)^{\beta\theta^2}.
\end{equation}
Finally, by using  the upper bound \eqref{eq:10} and provided  $C_0$ large enough, we obtain
\begin{equation}
  \label{eq:7}
  \sum_{i=-C+1}^{C-1}\mathbb P_{\beta,\lambda}[F_i]\le p_{\beta,\lambda}(K,\theta) \sum_{i=C_0-C}^{C-C_0}\Big(\frac{8}{C-|i|}\Big)^{\beta\theta^2}\le \tfrac 1{100}\,p_{\beta,\lambda}(K,\theta).
\end{equation}
Plugging \eqref{eq:7} and \eqref{eq:1a} in \eqref{eq:6} concludes the proof.
\end{proof}

\begin{proof}[Proof of Theorem~\ref{thm:main perco}(i)]
  Fix $\beta>1$. Let $\theta_\infty \in(\tfrac34,1)$ such that $\beta\theta_\infty^2>1$. Choose $\theta_1<1$ and $C_1\ge C_0(\beta,\theta_\infty)$ (where $C_0$ is provided by  Lemma~\ref{lem:1}) such that the sequences
\[
\begin{cases}C_{{n+1}}={(n+1)}^{3} C_1,\\
\theta_{n+1}:=\theta_n- {\tfrac{C_0}{C_{n+1}}},\end{cases}\quad\text{ for {$n\ge1$}}
\]
satisfy $\theta_n\ge\theta_\infty$ for every $n\ge1$.
Now, set $\lambda>0$ so large that 
\[
  p_{\beta,\lambda}(C_1,\theta_1)\le \mathbb P_{\beta,\lambda}[\exists \{x,x+1\}\subset B_{C_1}\text{ closed}]\le C_1e^{-\lambda}\le \frac1{1000 C_1^2}
\]
and consider the sequence of scales defined by
\begin{equation}
  \label{eq:13}
  \begin{cases}
    K_1=C_1,& \\
    K_{n+1}=C_{n+1} K_n&n\ge1.
  \end{cases}
\end{equation}
(Note that it gives  $K_n=(n!)^3 C_1^n$ for all $n\geq 1$).
Applying Lemma~\ref{lem:1} to $(\lambda,\beta,\theta_{n},C_n,K_n)$, we see that the sequence $u_{n}:=p_{\beta,\lambda}(K_n,\theta_n)$ satisfies 
\[
  \forall n\ge 1,\qquad u_{n+1}\le \tfrac1{100} u_n+2 C_{{n+1}}^2 u_n^2.\]
By multiplying the equation above by $C_{n+1}^2$ and using $C_{n+1}\le 4 C_n$, we find 
  \[
    \forall n\ge 1,\qquad C_{n+1}^2u_{n+1}\le \tfrac{16}{100} C_n^2u_n+512 (C_n^2 u_n)^2.\]
 
By induction, we obtain that $C_n^2 u_n\le  \tfrac1{1000} $ for every $n\ge 1$, and  therefore,
\[
\mathbb P_{\beta,\lambda}[B_{K_n}\text{ $\theta_n$-good}]\ge 1-\tfrac1{1000}C_n^{-2}\ge \tfrac12.\] 
First using the estimate above and then  translation invariance, we get that   for every~$n\ge1$, 
\begin{align}
  \label{eq:14}
  \tfrac34 K_n& \le \mathbb E[\tfrac 32K_n \cdot {\mathbf 1_{\{B_{K_n}\text{ is $\frac 3 4$-good\}}} }]\\
  &\le\mathbb E[|\mathbf C(B_{K_n})|\cdot {\mathbf 1_{\{B_{K_n}\text{ is $\frac 3 4$-good\}}} }]\\
  &\le 2K_n \mathbb P_{\beta,\lambda}[0\text{ is in a cluster of size at least $\tfrac 3 2 K_n$}].
\end{align}
Dividing both sides by $2K_n$, we obtain
\[
\mathbb P_{\beta,\lambda}[0\text{ is in a cluster of size at least $\tfrac 3 2  K_n$}]\ge \tfrac 38,
\]
which by measurability implies that the probability that 0 is connected to infinity is larger than or equal to $\frac 3 {{8}}$.
\end{proof}

\begin{remark}\label{rmk:1}
When considering $J_{i,j}=1/|i-j|^s$ with $s\in(1,2)$, the estimate in \eqref{eq:25} becomes of the order of $\exp[-c(\beta,\theta)((C-|i|)K)^{2-s}]$ and one can easily deduce the existence, for every $\beta>0$, of $\lambda=\lambda(\beta)>0$ large enough {so} that percolation occurs. Let us remark that in this case $p_{\beta,\lambda}(K_n,\theta_n)$ decays stretched-exponentially fast in $K_n$ (while it decays polynomially fast in the case of $s=2$).
\end{remark}
\subsection{Proof of Theorem~\ref{thm:main perco}(ii).}

We say that a block  $B^i_{3K}$ is \emph{$K$-crossed} if  there  exist $x< 3Ki-3K$ and $y\ge 3Ki+3K$ such that  $x$ is connected to $y$ using open edges of length at most $K$. Introduce
\[\overline p_{\beta,\lambda}(K)=1-\mathbb P[B_{3K}\text{ is $K$-crossed}].\] The proof is based on the following inequality (a similar inequality was obtained for another quantity in \cite{DMT20}).

\begin{lemma}\label{lem:2}
Let $\beta,\lambda,\theta>0$ such that $\beta\theta^2<1$ and $\theta(\beta,\lambda)<\theta$. Then, there exists $C_0=C_0(\beta,\lambda,\theta)$  such that for every integers $C,K\ge C_0$,
\begin{equation}\label{eq:17}
\overline p_{\beta,\lambda}(CK)\ge  C^{1-\beta\theta^2}\min\{\overline p_{\beta,\lambda}(K),C^{-1}\}.
\end{equation}
\end{lemma}

\begin{proof}Let us fix $\beta,\lambda,\theta$ and drop them from the notation. Fix $R\ge 1$ and $K\ge 2R$ such that 
\begin{equation}\label{eq:ojs}
\mathbb P[0\leftrightarrow \mathbb Z\setminus B_R]^2+|B_R|^2e^{-\beta (K-2R)}<\theta^2.
\end{equation}
We consider the percolation process restricted to the box $B_{CR}$, on consider only the edges of length at most $CR$. We say that an edge $\{x,y\}$ is
\begin{itemize}
\item[-] \emph{short} if $|y-x|\le K$,
\item[-] \emph{long} if $K<|y-x|\le CK$. 
\end{itemize}

Call an edge $\{x,y\}$ a {\em bridge} if it is open and in $\omega\setminus\{x,y\}$, {both} $x$ and $y$ are connected to distance $R$ around them.  

 Call a $3K$-block $B^i_{3K}$ {\em bridged} if
\begin{itemize}
\item there exists a bridge $\{x,y\}$ with $x< 3K(i-1)$ and $y\ge3K(i+1)$, or
\item there exists a long open edge adjacent to one point of $[3Ki-5K,3Ki+5K)$.
\end{itemize}
(otherwise it is said to be  \emph{unbridged}). The second condition in the definition of bridges may look unnatural at this point but it will be important in the proof, to ensure an independence property  between the set of bridged blocks and the set of $K$-crossed blocks.

Let $\mathbf U= \mathbf U(\omega)$ be the set  of unbridged $3K$-blocks $B^i_{3K}\subset B_{CK}$ with $i$ {\em divisible by} $3$ (note that the blocks are subsets of the box $B_{CK}$ and not $B_{3CK}$, and two such blocks are at distance at least $K+1$ of each other). 

Assume that the block $B_{3CK}$ is $CK$-crossed, then all of the $3K$-blocks in $B_{CK}$ must be either bridged or $K$-crossed. In particular, all the $3K$-blocks $B\in \mathbf U$ must be $K$-crossed.  This implies that 
\begin{equation}
  \label{eq:12}
  1-\overline p_{\beta,\lambda}(CK)\le  \mathbb P_{\beta,\lambda}\big[\forall B\in \mathbf U,\, B \text{ is $K$-crossed}\big].
\end{equation}

Now, we consider the sigma-algebra $\mathcal F$ generated by
  \begin{itemize}
  \item the status of all the long edges,
  \item the status of all short edges at distance  $\le R$ from an open long edge. 
  \end{itemize}
  The information of $\mathcal F$ can be revealed by the following two-step procedure: First, reveal the status of all the long edges, and second reveal the status of all the short edges at distance at most $R$ from the open long edges.

First, observe that the set $\mathbf U$ of unbridged blocks is measurable with respect to $\mathcal F$, and conditionally on $\mathcal F$, whether a block $B\in\mathbf U$ is $K$-crossed or not is independent of the other blocks in $\mathbf U$ (since they are at a distance at least $K$ of each other). Therefore, the right hand side in \eqref{eq:12} is equal to
\begin{equation}
  \label{eq:16}
\mathbb E_{\beta,\lambda}\big[\prod_{B\in\mathbf U}\mathbb P_{\beta,\lambda}[B \text{ is $K$-crossed}\: |\: \mathcal F]\big].
\end{equation}

Also,   the conditioning on $\mathcal F$ does not bring any  information on the fact that a $3K$-block $B\in \mathbf U$ is $K$-crossed (since $K> 2R$).  Indeed, after the two-step procedure described above, no short edge adjacent to $B$ has been revealed.  Hence, each term in the product above is equal to $1-\overline p_{\beta,\lambda}(K)$, and we get that
\begin{align}\label{eq:5}
  \overline p_{\beta,\lambda}(CK)&\ge 1-\mathbb E[(1-\overline p_{\beta,\lambda}(K))^{|\mathbf U|}].
\end{align}
Setting $t=\min\{\overline p_{\beta,\lambda}(K),C^{-1}\}$, we have 
\begin{equation}
  \label{eq:15}
  (1-\overline p_{\beta,\lambda}(K))^{|\mathbf U|}\le (1-t)^{|\mathbf U|}\le e^{-t|\mathbf U|}\le 1-e^{-1}t|\mathbf U|,
\end{equation}
where the last inequality uses $t|\mathbf U|\le1$. Taking the expectation and plugging it in \eqref{eq:5}, we deduce that
\begin{align}
\overline p_{\beta,\lambda}(CK)\ge  e^{-1} \mathbb E[{{|\mathbf U|}}]\min\{\overline p_{\beta,\lambda}(K),C^{-1}\}.
\end{align}
To conclude, it remains to bound  $\mathbb E[{|\mathbf U|}]$ {from below}. We do it by summing on $i$ the following estimate for $3K$-blocks $B^i_{3K}\subset B_{CK}$, 
\begin{align}
\mathbb P[B^i_{3K}\text{ unbridged}]
  &\ge  \prod_{\substack{x< 3K (i-1),\\ y\ge 3K(i+1),\\ y-x\le CK}}\!\!\!\!\!\!\mathbb P[\{x,y\}\text{ not a bridge}] \prod_{\substack{x\in [3Ki-5K,3Ki+5K),\\ K\le |y-x|\le CK}}\!\!\!\!\!\!\mathbb P[\{x,y\}\text{ is closed}]
  \\ & \ge 9 C^{-\beta\theta^2},\label{eq:3}
\end{align}
where the first inequality is due to the FKG inequality, and the second to a sum-integral comparisons (together with the assumption that $C$ is large enough) using the following estimate
\begin{align}
\mathbb P[\{x,y\}\text{ is a bridge}]&\le
(1-e^{-\beta J_{x,y}})(\mathbb P[0\leftrightarrow \mathbb Z\setminus B_R]^2+|B_R|^2e^{-\beta (K-2R)})\nonumber\\
&<(1-e^{-\beta J_{x,y}})\theta^2.\label{eq:3a}
\end{align}
The first inequality is due to the fact that either there is an open edge in $\omega\setminus\{x,y\}$ between $[x-R,x+R)$ and $[y-R,y+R)$, or the two events are independent. The second is due to the choice of $R$ given by \eqref{eq:ojs}. The strict inequality in \eqref{eq:3a} is important  in the computation leading to \eqref{eq:3}: it allows us to ``absorb'' all the constants by a sufficiently large choice of $C$.
\end{proof}

\begin{proof}[Proof of Theorem~\ref{thm:main perco}(ii)] Fix $\beta,\lambda,\theta>0$ such that  $\theta(\beta,\lambda)<\theta$ and $\beta\theta^2<1$. Let $C_0$ as in Lemma~\ref{lem:2}, and pick $C \ge C_0$ such that \[ C^{1-\beta\theta^2}\ge 1 \quad\text{and}\quad \overline p_{\beta,\lambda}(C_0)\ge C^{-1}.\]
  Setting $K_n:=C^nC_0$ ($n\ge0$), \eqref{eq:17} applied to $K=K_n$ and $C$ implies that for every $n\ge 0$,
  \begin{equation}
    \label{eq:18}
    \overline p_{\beta,\lambda}(K_{n+1})\ge \min( \overline p_{\beta,\lambda}(K_{n}),C^{-1}).
  \end{equation}
By induction, we deduce that  $\overline p_{\beta,\lambda}(K_n)\ge C^{-1}$ for every $n\ge 1$. 

Now, let $A(K_n)$ be the event that there exists $x\in B_{3K_n}$ connected to $y\notin B_{9K_n}$, and $B(K_n)$ be the event that all the edges of length strictly larger than $K_n$ with one endpoint in $B_{9K_n}$ are closed. Notice that if $B(K_n)$ occurs and neither $B^{-2}_{3K_n}$ nor $B^{2}_{3K_n}$ is  $K_n$-crossed, then $A(K_n)$ does not occur. Hence, by independence, we have that 
\begin{align}\label{eq:4}
\mathbb P_{\beta,\lambda}[B(K_n)]\overline p_{\beta,\lambda}(K_n)^2\le 1-\mathbb P[A(K_n)].
\end{align}
Since $\mathbb P_{\beta,\lambda}[B(K_n)]\ge c_1(\beta)>0$ (by a computation very similar to \eqref{eq:3}), we deduce that for every $n\ge0$,
\[
\mathbb P_{\beta,\lambda}[A(K_n)]\le 1-c_1(\beta)/C^2.
\]
{We obtained the above estimate by assuming $\theta(\beta,\lambda) < \theta$ with $\beta \theta^2 <1$. We see from this estimate that it is not possible to also have $\theta(\beta,\lambda)>0$. Indeed, otherwise, this would contradict} measurability since the probability that there exists $x\in B_{K_n}$ connected to infinity, which is itself included in $A(K_n)$, would have a probability tending to 1 in this case.\end{proof}

\section{Long-range Fortuin-Kasteleyn percolation and its applications to the Ising and Potts models}

\subsection{Statement of the results.} Here, we define the Fortuin-Kasteleyn percolation \cite{For70,ForKas72} (we also refer to \cite{Gri06,Dum17a} for general background on FK percolation). Let  $S\subset T$ be two finite subsets of $\mathbb Z$, let $\xi$ be a partition of the vertices $T\setminus S$. The FK percolation measure  on edges included in $T$ with at least one endpoint in $S$, with boundary conditions (b.c.) $\xi$, is  defined by the formula
\[
\mathbb P_{S,T,\beta,\lambda,q}^\xi[\omega]=\frac{q^{k(\omega^\xi)}}Z \prod_{\{i,j\}\subset T:\{i,j\}\cap S\ne \emptyset} p_{i,j}(\beta,\lambda)^{\omega_{i,j}}(1-p_{i,j}(\beta,\lambda))^{1-\omega_{i,j}},
\] 
where $\omega_{i,j}=1$ if $\{i,j\}$ is open and $0$  if it is closed, $\omega^\xi$ is the graph obtained from $\omega$ by wiring all the vertices outside $S$ belonging to the same element of the partition $\xi$. Let $\xi=1$ (resp.~$\xi=0$) be  the wired (resp.~free) boundary conditions corresponding to the partitions equal to $\{T\setminus S\}$ (resp.~only singletons). 

Given a partition $\xi$ of $\mathbb Z\setminus S$, define $\mathbb P_{S,\beta,\lambda,q}^\xi$ as the limit of the measure $\mathbb P_{S,B_K,\beta,\lambda,q}^{\xi_K}$ as $K$ tends to infinity, where $\xi_K$ denotes the partition induced by $\xi$ on $B_K\setminus S$. Let $\mathbb P_{\beta,\lambda,q}^1$ be the measure on $\mathbb Z$ defined as the limit as $K$ tends to infinity of the measures $\mathbb P_{B_K,\beta,\lambda,q}^1$ and $\theta(q,\beta,\lambda)$ be the $\mathbb P_{\beta,\lambda,q}^1$-probability that 0 is connected to infinity by a path of open edges\footnote{The proof that these limits exist for any partition $\xi$ of $\Z$ proceeds as usual by monotonicity.}.


\begin{theorem}\label{thm:main FK}
For $q\ge1$,
\begin{itemize}
\item[(i)] For $\beta>1$, there exists $\lambda<\infty$ large enough {so} that $\theta(q,\beta,\lambda)>0$.
\item[(ii)] For $\beta,\lambda>0$, $\theta(q,\beta,\lambda)>0$ implies $\beta\theta(q,\beta,\lambda)^2\ge 1$.
\end{itemize}
\end{theorem}

The result above covers a certain number of results, including \cite{NewSch} for (i) and  Aizenman, Chayes, Chayes and Newman \cite{aizenman1988} for (ii). In its current form, the result (i) corresponds to a paper of  Imbrie and Newman  \cite{ImbrieNewman}.

Using the coupling between FK percolation and Potts models (see e.g.~\cite{Gri06}), the previous theorem has the following corollary for the long-range 1D Ising and Potts models (see the beginning of the paper for the corresponding history and references). We do not define the models  here and simply introduce the parameter $m(q,\beta,\lambda)$ corresponding to the magnetization of the Potts model.
\begin{corollary}\label{cor:main FK}
Fix an integer $q\ge2$,
\begin{itemize}
\item[(i)] For $\beta>1$, there exists $\lambda<\infty$ large enough so that $m(q,\beta,\lambda)>0$.
\item[(ii)] For $\beta,\lambda>0$,  $m(q,\beta,\lambda)>0$ implies $\beta m(q,\beta,\lambda)^2\ge 1$. \end{itemize}
\end{corollary}

\subsection{Proof of Theorem~\ref{thm:main FK}(i).}

The proof is very similar to the proof of Theorem~\ref{thm:main perco}(ii) and we simply explain how the proof is modified. 
 Define
\[
p_{q,\beta,\lambda}(K,\theta):=\max_{\xi\text{ b.c. on }\mathbb Z\setminus B_K}\,\mathbb P_{K,\beta,\lambda,q}^\xi[B_K\text{ is $\theta$-bad}]
\]
(note that by monotonicity it is achieved for free boundary conditions $\xi=0$ but we will not use this fact). 
In the proof of Lemma~\ref{lem:1}, all the deterministic observations are the same. Also, \eqref{eq:1a} can be obtained in the same way as before using the spatial Markov property (since $p_{q,\beta,\lambda}(K,\theta)$ is expressed in terms of the maximum over boundary conditions). 

The only step that requires care is the proof of \eqref{eq:25}. Indeed, the first difference is that the states of the edges $\{x,y\}$ with $x\in\mathbf C^-$ and $y\in\mathbf C^+$ are not independent of the conditioning on $\mathbf C^-$ and $\mathbf C^+$, and second that the state of edges are not independent of each other. Still, we now show that the estimate \eqref{eq:25} holds for every $q>0$. 

Condition on $\mathbf C^-$, $\mathbf C^+$, and every edge that is not linking $x\in \mathbf C^-$ and $y\in\mathbf C^+$. Consider the graph $G$ composed of vertices in $\mathbf C^-\cup\mathbf C^+$ and edges between $x\in\mathbf C^-$ and $y\in\mathbf C^+$ and observe that the previous conditioning does not reveal the state of these edges. Let $\xi$ be the boundary condition induced by this conditioning and let $\mathbb P^\xi_G$ be the associated FK percolation on $G$. Note that all the vertices in $\mathbf C^-$ (resp.~$\mathbf C^+$) are wired together. 

At this stage it is unclear whether the vertices in $\mathbf C^-$ are wired to those in $\mathbf C^+$ or not by $\xi$. If they are, then the probability that each edge $\{x,y\}$ is closed is $e^{-\beta J_{x,y}}$ independently of the other edges and the same computation as in the Bernoulli case holds true. Otherwise, consider an intermediate constant $R\in[1,C)$ and define $A$ (resp.~$B$) as the event that all the edges in $G$ with $\{x,y\}\subset B_{RK}$ are closed (resp.~all the remaining edges of $G$). 

{
Notice that 
\begin{align}\label{e.AB}
\mathbb P_G^\xi[A\cap B]&\le  \frac{\mathbb P_G^\xi[B|A^c]}{\mathbb P_G^\xi[A^c|B]}\,.
\end{align}
Let us analyze first the numerator $P_G^\xi[B|A^c]$. Thanks to the conditioning on $A^c$, we may work with products over edges 
of 
 $e^{-\beta J_{x,y}}$ instead of the less convenient quantity 
\begin{align}\label{e.vee}
e^{-\beta J_{x,y}} \vee 
\frac{q\, e^{-\beta J_{x,y}}}
{1-e^{-\beta J_{x,y}} + q\, e^{-\beta J_{x,y}} }\,
\end{align} which would arise from the probability $\frac{1-p}{p+(1-p)q}$ that an edge $e=\{x,y\}$ is closed knowing that its endpoints are not connected in $\omega_{|G\setminus e}$. This is the reason why we introduced the event $A$ on the mesoscopic scale $RK \ll CK$. Recall that the argument assumes that the $K$-blocks $B_k^j, j\notin \{i-1,i,i+1\}$ are {\em $\theta$-good} which implies that  the density of $\mathbf C^-$ and $\mathbf C^+$ are at least $\theta$ on both sides. The above observation that the weight on each edge is $e^{-\beta J_{x,y}}$ knowing $A^c$ now implies (provided $R$ is chosen large enough, and then $C$ even larger), by following the same sum/integral analysis as for the estimate \eqref{eq:25}, that
\begin{align}\label{e.num}
P_G^\xi[B|A^c] & \leq O(1) \exp\Big[-\beta \theta^2  \sum_{r,s=RK}^{CK}\frac1{(r+s-1)^2}\Big]   \leq O(1) \Big( \frac {R} {C} \Big)^{\beta \theta^2}\,.
\end{align}
Now, for the denominator $\mathbb P_G^\xi[A^c|B]$, start by noticing that if $R$ is large enough and if $q\geq 1$, then the FK weight in~\eqref{e.vee}  for any $e=\{x,y\}$ in $G$ is at most
\begin{align*}\label{}
\frac{q\, e^{-\beta J_{x,y}}}
{1-e^{-\beta J_{x,y}} + q\, e^{-\beta J_{x,y}} }& = 
e^{-\beta J_{x,y}} \frac 1 {1-\tfrac{q-1}q (1-e^{-\beta J_{x,y}})} \\
& \leq 
 e^{-\beta J_{x,y}} e^{\frac{q-1/2}q \beta  J_{x,y}} \\
 & =  e^{-\frac \beta {2q} J_{x,y}}\,.
\end{align*}
If on the other hand $0<q\leq 1$, this is simpler as the weights in~\eqref{e.vee} are then smaller than $e^{-\beta J_{x,y}}$. 

As such, modulo the same analysis as for estimate~\eqref{eq:25}, this gives us an upper bound for $\mathbb P_G^\xi[A|B]$ of order
\begin{align*}
\mathbb P_G^\xi[A|B] & \leq O(1) \exp\Big[-\beta \theta^2 (1\wedge \frac 1{2q}) \sum_{r,s=K}^{RK}\frac1{(r+s-1)^2}\Big] \\
& \leq  O(1)  \exp(- \Omega(1) \log R) \leq \frac 1 2\,,
\end{align*}
if $R$ is chosen large enough. 
}
{
By plugging this estimate together with \eqref{e.num} into \eqref{e.AB} and choosing $R=C^{\delta}$ with $C$ large enough, this gives us 
\begin{align*}\label{}
\mathbb P_G^\xi[A\cap B]& \leq O(1) C^{-\beta \theta^2(1-\delta)}\,,
\end{align*}
for any small exponent $\delta$. This ends the proof of the analog of estimate~\eqref{eq:25} for FK percolation. (Note that we have written the proof for the central block $B_K^0$, but the same analysis would give an upper bound of $O(1) \left( C-|i| \right)^{-\beta \theta^2(1-\delta)}$ for the block $B_k^i$). 
}
\begin{remark}
Notice that {remarkably}, the proof of {Theorem~\ref{thm:main FK}(i)}  works for every $q>0$.
\end{remark}

\subsection{Proof of Theorem~\ref{thm:main FK}(ii).} 

Define
\[
\overline p_{q,\beta,\lambda}(K,\theta):=1-\max_{\substack{S\subset B_{3K}\\\xi\text{ b.c. on }\mathbb Z\setminus B_{3K}}} \,\mathbb P_{S,\beta,\lambda,q}^\xi[B_{3K}\text{ is $K$-crossed}].
\] 
(the maximum is achieved for $S=B_{3K}$ and for wired boundary conditions $\xi=1$ but we will not use this fact here). 
Then, the proof of Lemma~\ref{lem:2} is the same except in two places. First, the states of edges in different unbridged boxes are not independent anymore so to derive \eqref{eq:5}, we replace independence by the spatial Markov property and the fact that $\overline p_{q,\beta,\lambda}(K,\theta)$ is defined as a minimum over boundary conditions. 

The rest of the proof is the same (we can use the FKG inequality since $q\ge1$), except in the proof of \eqref{eq:3a}.  There, we use that conditioned on the configuration outside of $\{x,y\}$, the probability that $\{x,y\}$ is open is smaller than $1-e^{-\beta J_{x,y}}$ (since $q\ge1$), as well as the observation that  we can pick $R$ such that 
\[
\max_\xi \mathbb P_{B_R,\beta,\lambda,q}^\xi[0\leftrightarrow \mathbb Z\setminus B_R]^2+|B_R|^2e^{-\beta (K-2R)/q}\le \theta^2
\] 
since the maximum is reacher for $\xi=1$ and that the quantity converges to $\theta(q,\beta,\lambda)$ as $R$ tends to infinity.

\subsection*{Acknowledgements} The first author is funded by the ERC CriBLaM, the Swiss FNS and the NCCR SwissMap. The research of the second author is supported by the ERC grant LiKo 676999. 
The third author is supported by the European Research Council (ERC) under the European Union’s Horizon 2020 research and innovation program (grant agreement No 851565) and by the NCCR SwissMap.

\bibliographystyle{alpha}

\end{document}